\documentclass[hidelinks,onefignum,onetabnum]{siamart220329}

\usepackage{lipsum}
\usepackage{amsfonts}
\usepackage{graphicx}
\usepackage{epstopdf}
\usepackage{algorithmic}
\ifpdf
  \DeclareGraphicsExtensions{.eps,.pdf,.png,.jpg}
\else
  \DeclareGraphicsExtensions{.eps}
\fi

\usepackage{liyuancao}
\DeclareMathOperator{\conv}{conv}
\DeclareMathOperator{\diag}{diag}

\usepackage{tikz}

\usepackage{mathtools}
\newlength{\leftstackrelawd}
\newlength{\leftstackrelbwd}
\def\leftstackrel#1#2{\settowidth{\leftstackrelawd}%
{${{}^{#1}}$}\settowidth{\leftstackrelbwd}{$#2$}%
\addtolength{\leftstackrelawd}{-\leftstackrelbwd}%
\leavevmode\ifthenelse{\lengthtest{\leftstackrelawd>0pt}}%
{\kern-.5\leftstackrelawd}{}\mathrel{\mathop{#2}\limits^{#1}}}

\newsiamremark{remark}{Remark}
\newsiamremark{hypothesis}{Hypothesis}
\crefname{hypothesis}{Hypothesis}{Hypotheses}
\newsiamthm{claim}{Claim}

\headers{Sharp Error Bounds for Linear Interpolation}{Liyuan Cao, Zaiwen Wen, and Ya-Xiang Yuan}

\title{Some Sharp Error Bounds for Multivariate Linear Interpolation and Extrapolation\thanks{Submitted to the editors on September 26, 2022.
\funding{This research was supported by NSFC grants 11831002.}}}

\author{Liyuan Cao \thanks{Beijing International Center for Mathematical Research, Peking University, Beijing, China
(\email{caoliyuan@bicmr.pku.edu.cn}, \email{wenzw@pku.edu.cn}).}
        \and Zaiwen Wen \footnotemark[2] 
        \and Ya-xiang Yuan \thanks{State Key Laboratory of Scientific and Engineering Computing, Academy of Mathematics and Systems Science, Chinese Academy of Sciences, Beijing, China (\email{yyx@lsec.cc.ac.cn})}}

\usepackage{amsopn}

\ifpdf
\hypersetup{
  pdftitle={Some Sharp Error Bounds for Multivariate Linear Interpolation and Extrapolation},
  pdfauthor={Liyuan Cao and Zaiwen Wen}
}
\fi

\begin{document}

\maketitle

\begin{abstract}
  We study in this paper the function approximation error of linear interpolation and extrapolation. 
Several upper bounds are presented along with the conditions under which they are sharp. 
All results are under the assumptions that the function has Lipschitz continuous gradient and is interpolated on an affinely independent sample set. 
Errors for quadratic functions and errors of bivariate linear extrapolation are analyzed in depth. 
\end{abstract}



\section{Introduction} \label{sec:intro}
Polynomial interpolation is one of the most basic techniques for approximating functions and plays an essential role in applications such as finite element methods and derivative-free optimization. 
This led to a large amount of literature concerning its approximation error. 
This paper contributes to this area of study by providing some sharp bounds on the function approximation error of linear interpolation. 
Specifically, given a function $f: \R^n \rightarrow \R$ and an affinely independent sample set $\cY = \{y_1,y_2,\dots,y_{n+1}\} \subset \R^n$, one can find a unique linear function $m: \R^n \rightarrow \R$ such that $m(y_i) = f(y_i)$ for all $y_i \in \cY$. 
Certainly, the approximation error $m(y_0) - f(y_0)$ depends on the function $f$, the sample set $\cY$, and the point $y_0\in\R^n$, where the error is measured. 
We investigate in this paper the sharp upper bound on $|m(y_0) - f(y_0)|$ for any given affinely independent $\cY$ and $y_0$ under the assumption that $f \in C_L^{1,1}(\R^n)$, where $C_L^{1,1}(\R^n)$ is the set of differentiable functions defined on $\R^n$ with Lipschitz continuous gradient, i.e., 
\begin{equation} \label{eq:Lipschitz}
    \|\nabla f(u_1) - \nabla f(u_2)\| \le L \|u_1 - u_2\| \quad \text{for all } (u_1,u_2) \in \R^n \times \R^n, 
\end{equation}
where $L>0$ is the Lipschitz constant and the norms are Euclidean. 

The sharp bound on $|m(y_0) - f(y_0)|$ is already discovered and proved in \cite{waldron1998error} for the case when $y_0\in\conv(\cY)$, the convex hull of $\cY$. 
However, in applications like model-based derivative-free optimization, where linear interpolation is employed to approximate the black-box objective function \cite{powell1994direct, DFO_book}, the model $m$ is used more often than not to estimate the function at a point outside $\conv(\cY)$ (sometimes referred to as linear extrapolation). 
As will be shown later in this paper, establishing the sharp error bound in this case is much more difficult, because consideration of the geometry of $\cY$ and $y_0$ is required, while it is not when $y_0\in\conv(\cY)$. 
Furthermore, having obtuse angles at the vertices of the simplex $\conv(\cY)$ complicates the problem significantly, and it is not a viable approach to linearly transform the input space so that $\conv(\cY)$ is non-obtuse, since condition \cref{eq:Lipschitz} will also be changed at the same time. 

The function approximation error of univariate ($n=1$) interpolation using polynomials of any degree is already well-studied, and the results can be found in classical literature such as \cite{davis1975book}. 
If a $(d+1)$-times differentiable function $f$ defined on $\R$ is interpolated by a polynomial of degree $d$ on $d+1$ unique points $\{y_1, y_2, \dots, y_{d+1}\} \subset \R$, then the resulting polynomial $m_d$ has the approximation error 
\begin{equation} \label{eq:Cauchy remainder}
    f(y_0) - m_d(y_0) = \frac{(y_0-y_1)(y_0-y_2)\cdots(y_0-y_{d+1})}{(d+1)!} f^{(d+1)}(\xi) \quad \text{for all } y_0 \in \R
\end{equation} 
for some $\xi$ with $\min(y_0,y_1,,\dots,y_{d+1}) < \xi < \max(y_0,y_1,\dots,y_{d+1})$. 
Unfortunately this result cannot be extended to multivariate interpolation directly, even if the polynomial is linear ($d=1$). 

The function approximation error of multivariate ($n>1$) polynomial interpolation has been studied by researchers from multiple research fields. 
Motivated by their application in finite element methods, errors in both Lagrange and Hermite interpolation with polynomials of any degree were analyzed in \cite{ciarlet1972general}. 
As a part of an effort to develop derivative-free optimization algorithms, a bound on the error of quadratic interpolation was provided in \cite{powell2001lagrange}. 
The sharp error bound for linear interpolation was found by researchers of approximation theory for the case when $y_0 \in \conv(\cY)$ using the unique Euclidean sphere that contains $\cY$ \cite{waldron1998error}. 
Following \cite{waldron1998error}, a number of sharp error bounds were derived in \cite{stampfle2000optimal} for linear interpolation under several different smoothness or continuity assumptions in addition to \cref{eq:Lipschitz}. 

Following these works, we investigate in this paper the sharp bound on the function approximation error of linear interpolation when $y_0\not\in\conv(\cY)$. 
Our approach involves treating the problem of finding the sharp bound as two optimization problems. 
This first one is to minimize an upper bound $z$: 
\begin{equation} \label{prob:P} \everymath{\displaystyle} \begin{array}{ll} 
    \min_z &z \\
    \text{s.t. } &z \ge |m(y_0) - f(y_0)| \quad \text{for all } f \in C_L^{1,1}(\R^n); 
\end{array} \end{equation}
and the second one is to maximize the error with respect to functions in $C_L^{1,1}(\R^n)$: 
\begin{equation} \label{prob:D} \everymath{\displaystyle} \begin{array}{ll} 
    \max_f &|m(y_0) - f(y_0)| \\
    \text{s.t. } &f \in C_L^{1,1}(\R^n).  
\end{array} 
\end{equation}
It is not realistic to directly solve these two optimization problems due to the difficulty in handling the inclusion $f \in C_L^{1,1}(\R^n)$, but we will use them as a guide to our results. 
Our main contributions are as follows. 
\begin{enumerate}
    \item For functions in $C_L^{1,1}(\R^n)$, an upper bound on the function value difference between any two points is established using the gradients at both points. 
    \item An upper bound on the function approximation error of linear interpolation is derived and proved to be sharp when $y_0\in\conv(\cY)$ and when $y_0\not\in\conv(\cY)$ if certain condition is met. 
    \item The largest function approximation error achievable by quadratic functions in $C_L^{1,1}(\R^n)$ is found and the condition under which it is an upper bound on the error achievable by all functions in $C_L^{1,1}(\R^n)$ is determined. 
    \item A formula is provided for the sharp bound on the function approximation error of bivariate ($n=2$) linear interpolation. 
\end{enumerate}

The paper is organized as follows. 
Our notation and the preliminary knowledge are introduced in \cref{sec:preliminaries}. 
In \cref{sec:phase1}, we improve an existing upper bound by first generalizing and then minimizing it. 
In \cref{sec:phase2}, we solve problem~\cref{prob:D} while limiting $f$ to quadratic functions. 
In \cref{sec:phase3}, we show how to calculate the sharp bound on function approximation error of bivariate linear interpolation. 
We conclude the paper in \cref{sec:discussion} by discussing our findings and the obstacles that prevented us from progressing further.



\section{Preliminaries} \label{sec:preliminaries}
For any vector $u$, we denote by $[u]_i$ its $i$th entry. 
For any matrix $U$, we denote by $[U]_{ij}$ the entry in its $i$th row and $j$th column. 
We denote by $\|\cdot\|$ the Euclidean norm. 
The inner product $\langle \cdot, \cdot \rangle$ is the summation of the entry-wise product, that is $\langle u_1, u_2 \rangle = \sum_i [u_1]_i [u_2]_i$ for any pair of vectors $(u_1,u_2)$, and $\langle U_1, U_2 \rangle = \sum_{ij} [U_1]_{ij} [U_2]_{ij}$ for any pair of matrices $(U_1,U_2)$. 

Let $e_i$ be the vector that is all 0 but having 1 as its $i$th entry. 
Let $Y \in \R^{(n+1)\times n}$ be the matrix such that its $i$th row $\displaystyle Y^T e_i = y_i-y_0$ for all $i\in\{1,2,\dots,n+1\}$. 
Define $\phi:\R^n \rightarrow \R^{n+1}$ as the function that for all $u\in\R^n$
\[ [\phi(u)]_i = \left\{ \begin{array}{ll}
    1 &\text{if } i=1,  \\
    [0pt] [u]_{i-1} &\text{if } i \in \{2,3,\dots,n+1\}. 
\end{array} \right.
\]
We denote by $\Phi$ the $(n+1)$-by-$(n+1)$ matrix such that $\Phi_{ij} = [\phi(y_i-y_0)]_j$ for all $(i,j) \in\{1,2,\dots,n+1\}^2$. 
Notice the affine independence of $\cY$ implies the nonsigularity of $\Phi$. 

Let $\ell \in \R^{n+1}$ be the {\it barycentric coordinate} of $y_0$ with respect to $\cY$, and $\ell_i = [\ell]_i$ for all $i \in \{1,2,\dots,n+1\}$. 
In addition, we define $\ell_0 = -1$. 
This coordinate system has the properties 
\begin{align}
\sum_{i=1}^{n+1} \ell_i f(y_i) &= m(y_0), \label{eq:Lagrange m}  \\
    \sum_{i=0}^{n+1} \ell_i &= 0, \label{eq:Lagrange 0} \\
    \sum_{i=0}^{n+1} \ell_i y_i &= 0, \label{eq:Lagrange Y} 
\end{align}
and can be calculated as $\ell = \Phi^{-T} \phi(\mathbf{0})$. 
It can be seen as the {\it Lagrange polynomials} evaluated at $y_0$, hence our choice of the symbol $\ell$. 

Without loss of generality, we assume the set $\cY = \{y_1,y_2,\dots,y_{n+1}\}$ is ordered in a way such that $\ell_1 \ge \ell_2 \ge \cdots \ge \ell_{n+1}$. 
We define the following two sets of indices: 
\begin{subequations} \begin{align} 
\cI_+ &= \{i\in \{0,1,\dots,n+1\}:~ \ell_i>0\} = \{1,2,\dots, |\cI_+|\}, \\
\cI_- &= \{i\in \{0,1,\dots,n+1\}:~ \ell_i<0\} = \{0, n+3-|\cI_-|, \dots, n+1\}. 
\end{align} \end{subequations}
Notice \cref{eq:Lagrange 0} implies $|\cI_+| \ge 1$, and it is possible for $n+3-|\cI_-| > n+1$, in which case $\cI_- = \{0\}$. 

We define 
\begin{equation} \label{eq:G}
    G = \sum_{i=0}^{n+1} \ell_i y_i y_i^T, 
\end{equation}
which has the property that for any $(u_1,u_2) \in \R^n\times\R^n$, 
\begin{equation} \label{eq:G recenter} \begin{aligned}
    \sum_{i=0}^{n+1} \ell_i (y_i-u_1) (y_i-u_2)^T 
    &= \sum_{i=0}^{n+1} \ell_i \left[y_iy_i^T - u_1y_i^T - y_iu_2^T + u_1u_2^T \right] \\
    &\leftstackrel{\cref{eq:Lagrange Y}}{=} \sum_{i=0}^{n+1} \ell_i \left[y_iy_i^T + u_1u_2^T\right] 
    \stackrel{\cref{eq:Lagrange 0}}{=} \sum_{i=0}^{n+1} \ell_i y_iy_i^T = G. 
\end{aligned} \end{equation}

It is well-known (see, e.g., section 1.2.2 of the textbook \cite{Nesterov_book}) that the inclusion $f\in C_L^{1,1}(\R^n)$ implies 
\begin{equation} \label{eq:Lipschitz quadratic}
    |f(u_2) - f(u_1) - \langle \nabla f(u_1), u_2 - u_1 \rangle| \le \frac{L}{2} \|u_2 - u_1\|^2 \text{ for all } (u_1,u_2) \in \R^n \times \R^n, 
\end{equation}
and that if $f$ is twice differentiable on $\R^n$, \cref{eq:Lipschitz} is equivalent to 
\begin{equation} \label{eq:Lipschitz Hessian}
    -LI \preceq \nabla^2 f(u) \preceq LI \text{ for all } u \in \R^n. 
\end{equation}
What is less well-known is that \cref{eq:Lipschitz quadratic} also implies \cref{eq:Lipschitz}. 
We show in the following proposition the equivalence among the conditions \cref{eq:Lipschitz,eq:Lipschitz quadratic,eq:Lipschitz stronger} for differentiable functions. 
The inequality \cref{eq:Lipschitz stronger} does not exist in the current literature to our best knowledge. 
\begin{proposition}
Assume $f:\R^n \rightarrow \R$ is differentiable. 
If \cref{eq:Lipschitz} holds, then for all $(u_1, u_2) \in \R^n \times \R^n$
\begin{equation} \label{eq:Lipschitz stronger} \begin{aligned}
    f(u_1) \le f(u_2) &+ \langle \frac{\nabla f(u_1) + \nabla f(u_2)}{2}, u_1 - u_2\rangle\\ 
    &+ \frac{L}{4} \|u_1-u_2\|^2 - \frac{1}{4L} \|\nabla f(u_1) - \nabla f(u_2)\|^2, 
\end{aligned} \end{equation} 
and vice versa. 
\end{proposition}
\begin{proof}
Let \cref{eq:Lipschitz} hold. 
Given any $(u_1,u_2)$, since \cref{eq:Lipschitz} implies \cref{eq:Lipschitz quadratic}, for any $u \in \R^n$ there are 
\begin{align*}
    f(u) &\stackrel{\cref{eq:Lipschitz quadratic}}{\ge} f(u_1) + \langle \nabla f(u_1), u-u_1 \rangle - \frac{L}{2} \|u-u_1\|^2,  \\
    f(u) &\stackrel{\cref{eq:Lipschitz quadratic}}{\le} f(u_2) + \langle \nabla f(u_2), u-u_2 \rangle + \frac{L}{2} \|u-u_2\|^2.  
\end{align*}
Then the right-hand side of the first inequality is less than or equal to the right-hand side of the second inequality: 
\[ f(u_1) - f(u_2) + \langle \nabla f(u_1), u-u_1 \rangle - \langle \nabla f(u_2), u-u_2 \rangle - \frac{L}{2}\left(||u-u_1||^2+||u-u_2||^2\right) \le 0. 
\]
Setting $u=(u_1+u_2)/2 + (\nabla f(u_1) - \nabla f(u_2)) / (2L)$ yields \cref{eq:Lipschitz stronger}. 
On the other hand, \cref{eq:Lipschitz} can be obtained by adding together \cref{eq:Lipschitz stronger} and \cref{eq:Lipschitz stronger} with $u_1$ and $u_2$ reversed, then multiplying both sides of the inequality by $2L$ and taking their square roots. 
\end{proof}

\section{An Improved Upper Bound} \label{sec:phase1}
The results in \cite{ciarlet1972general} and \cite{powell2001lagrange} are obtained by comparing $f$ against its Taylor expansion at $y_0$. 
We generalize their approach in \cref{thm:phase1} by using the Taylor expansion of $f$ at an arbitrary $u \in \R^n$. 
\begin{theorem}\label{thm:phase1}
Assume $f \in C^{1,1}_L(\R^n)$. 
Let $m$ be the linear function that interpolates $f$ at any set of $n+1$ affinely independent vectors $\cY = \{y_1,\dots,y_{n+1}\}\subset \R^n$. 
The function approximation error of $m$ at any $y_0\in\R^n$ is bounded as 
\begin{equation} \label{eq:phase1 u}
    |m(y_0) - f(y_0)| \le \frac{L}{2} \sum_{i=0}^{n+1} |\ell_i| \|y_i-u\|^2, 
\end{equation}
where $u$ can be any vector in $\R^n$. 
\end{theorem}

\begin{proof}
By \eqref{eq:Lipschitz quadratic}, we have for any $u \in \R^n$ 
\begin{subequations} \label{phase1 set of inequalities} \begin{align}
\ell_i [f(y_i) - f(u) - \langle \nabla f(u), y_i-u \rangle] &\le \ell_i \frac{L}{2} \|y_i-u\|^2 \text{ for all } i\in \cI_+, \\
-\ell_i [-f(y_i) + f(u) + \langle \nabla f(u), y_i-u \rangle] &\le -\ell_i \frac{L}{2} \|y_i-u\|^2 \text{ for all } i\in \cI_-. 
\end{align} \end{subequations}
Now add all inequalities above together. 
The sum of the zeroth-order terms (with respect to $f$) on the left-hand sides of the inequalities in \eqref{phase1 set of inequalities} is 
\[ \sum_{i=0}^{n+1} \ell_i [f(y_i) - f(u)] 
\stackrel{\eqref{eq:Lagrange 0}}{=} \sum_{i=0}^{n+1} \ell_i f(y_i)
\stackrel{\eqref{eq:Lagrange m}}{=} m(y_0) - f(y_0). 
\]
The sum of the first-order terms is 
\[ \langle \nabla f(u), \sum_{i=0}^{n+1} \ell_i (u-y_i) \rangle
\stackrel{\eqref{eq:Lagrange 0}}{=} \langle \nabla f(u), \sum_{i=1}^{n+1} \ell_i y_i \rangle
\stackrel{\eqref{eq:Lagrange Y}}{=} 0. 
\]
The sum of the right-hand sides is $L/2 \sum_{i=0}^{n+1} |\ell_i| \|y_i-u\|^2$. 
Thus the sum of the inequalities in \eqref{phase1 set of inequalities} is \eqref{eq:phase1 u} when $m(y_0)-f(y_0) \ge 0$. 
If the inequalities in \eqref{phase1 set of inequalities} have their left-hand sides multiplied by $-1$, they would still hold according to \eqref{eq:Lipschitz quadratic}, and their summation would be \eqref{eq:phase1 u} for the $m(y_0)-f(y_0) < 0$ case. 
\end{proof}

Considering the right-hand side of \eqref{eq:phase1 u} is a convex function of $u$ defined on $\R^n$, in the spirit of problem \cref{prob:P}, we can minimize this upper bound with respect to $u$, resulting in \cref{eq:phase1}. 
\begin{corollary} \label{cor:phase1}
Under the setting of \cref{thm:phase1}, the function approximation error of $m$ at any $y_0\in\R^n$ is bounded as 
\begin{equation} \label{eq:phase1}
    |m(y_0) - f(y_0)| \le \frac{L}{2} \sum_{i=0}^{n+1} |\ell_i| \|y_i-w\|^2, 
\end{equation}
where
\[ w = \frac{\sum_{i=0}^{n+1} |\ell_i| y_i}{\sum_{i=0}^{n+1} |\ell_i|}. 
\]
\end{corollary}

When $f \in C^{1,1}_L(\R^n)$, the proof of Theorem 3.1 in \cite{waldron1998error} shows that 
\begin{equation} \label{eq:Waldron} 
    |m(y_0) - f(y_0)| \le \frac{L}{2} \sum_{i=0}^{n+1} \ell_i \|y_i\|^2, 
\end{equation}
and that \eqref{eq:Waldron} is sharp when $y_0 \in \conv(\cY)$, or equivalently when $\ell_i \ge 0$ for all $i\in\{1,2,\dots,n+1\}$. 
We show in \cref{thm:phase1 convex hull} that \eqref{eq:phase1} is the same as the sharp bound \eqref{eq:Waldron} when $y_0 \in \conv(\cY)$. 
\begin{theorem} \label{thm:phase1 convex hull}
When $y_0 \in \conv(\cY)$, the bounds \eqref{eq:phase1} and \eqref{eq:Waldron} are the same sharp bound. 
\end{theorem}
\begin{proof}
When $\ell_i \ge 0$ for all $i\in\{1,2,\dots,n+1\}$, 
\[ w = \frac{y_0 + \sum_{i=1}^{n+1} \ell_i y_i}{1 +\sum_{i=1}^{n+1} \ell_i}
\stackrel{\eqref{eq:Lagrange Y}\eqref{eq:Lagrange 0}}{=} \frac{y_0 + y_0}{1 + 1}
= y_0. 
\]
The following chain of equalities shows the equivalence between \eqref{eq:phase1} and \eqref{eq:Waldron}: 
\[ \begin{aligned}
    \sum_{i=0}^{n+1} |\ell_i| \|y_i-w\|^2 
    &= \sum_{i=0}^{n+1} \ell_i \|y_i-y_0\|^2 
    = \sum_{i=0}^{n+1} \ell_i \left(\|y_i\|^2 - 2\langle y_i,y_0 \rangle + \|y_0\|^2\right) \\
    &\leftstackrel{\eqref{eq:Lagrange 0}}{=} \sum_{i=0}^{n+1} \ell_i \left(\|y_i\|^2 - 2\langle y_i,y_0 \rangle \right) 
    \stackrel{\eqref{eq:Lagrange Y}}{=} \sum_{i=0}^{n+1} \ell_i \|y_i\|^2. 
\end{aligned} \]
The sharpness can be shown with the function $f(u) = \frac{L}{2} \|u\|^2 \stackrel{\eqref{eq:Lipschitz Hessian}}{\in} C^{1,1}_L(\R^n)$, which has 
\[ m(y_0) - f(y_0) 
\stackrel{\eqref{eq:Lagrange m}}{=} \sum_{i=0}^{n+1} \ell_i f(y_i) 
= \sum_{i=0}^{n+1} \ell_i \frac{L}{2} \|y_i\|^2. 
\]
\end{proof}

\Cref{thm:phase1 negative} shows \cref{eq:phase1} can also be sharp for extrapolation. 
\begin{theorem} \label{thm:phase1 negative}
The bound \eqref{eq:phase1} is sharp when there is only one positive entry in $\ell$. 
\end{theorem}
\begin{proof}
Due to our ordering of the points in $\cY$, we have $\cI_+ = \{1\}$ if there is only one positive entry in $\ell$. 
In this case,  
\[  w = \frac{2\ell_1 y_1 - \sum_{i=0}^{n+1} \ell_i y_i}{2\ell_1 -\sum_{i=0}^{n+1} \ell_i} 
    \stackrel{\eqref{eq:Lagrange 0}\eqref{eq:Lagrange Y}}{=} \frac{2\ell_1 y_1}{2\ell_1} 
    = y_1. 
\] 
The bound \eqref{eq:phase1} equals $L/2$ multiplies 
\[ \begin{aligned}
    \sum_{i=0}^{n+1} |\ell_i| \|y_i-w\|^2 
    &= - \sum_{i=0}^{n+1} \ell_i \|y_i-y_1\|^2 
    = \text{trace}\left( - \sum_{i=0}^{n+1} \ell_i (y_i-y_1)(y_i-y_1)^T \right) \\
    &\leftstackrel{\eqref{eq:G recenter}}{=} \text{trace}\left( - \sum_{i=0}^{n+1} \ell_i y_iy_i^T \right)
    =-\sum_{i=0}^{n+1} \ell_i \|y_i\|^2. 
\end{aligned} \]
Consider the function $f(u) = -\frac{L}{2} \|u\|^2 \stackrel{\eqref{eq:Lipschitz Hessian}}{\in} C^{1,1}_L(\R^n)$. 
We have 
\[m(y_0) - f(y_0) 
\stackrel{\eqref{eq:Lagrange m}}{=} \sum_{i=0}^{n+1} \ell_i f(y_i)  
= -\sum_{i=0}^{n+1} \ell_j \frac{L}{2}\|y_i\|^2, 
\]
which shows \eqref{eq:phase1} is sharp. 
\end{proof}

\Cref{fig:phase1} shows three sets of areas in which $y_0$ can locate relative to any set of affinely independent $\cY \subset \R^2$. 
The ordering of the points in $\cY = \{y_1,y_2,y_3\}$ in \cref{fig:phase1} (and all figures hereafter) is arbitrary and not determined by the values of $\ell$, which depends on the location of $y_0$. 
Geometrically, if $\ell_i$ is the only positive element in $\ell$ for some $i\in\{1,2,\dots,n+1\}$, then $y_0$ locates in the cone 
\[ \left\{u = y_i + \sum_{j=1}^{n+1} \alpha_j(y_i - y_j):~ \alpha_j \ge 0 \text{ for all } j \in \{1,2,\dots,n+1\} \right\}. 
\]
Specifically in \cref{fig:phase1 negative}, if $\ell_i$ is the only positive element in $\ell$ for some $i\in\{1,2,3\}$, then $y_0$ is in the shaded cone originated from $y_i$. 
\begin{figure}[tbhp]
     \centering
     \subfloat[\raggedright the area covered by \cref{thm:phase1 convex hull}]{\label{fig:phase1 hull}\resizebox{0.3\linewidth}{!}{\begin{tikzpicture}
\filldraw[black] (1,0) circle (2pt) node[anchor=north] {$y_1$}; 
\filldraw[black] (-0.3,1) circle (2pt) node[anchor=east] {$y_2$}; 
\filldraw[black] (-1.1,-0.5) circle (2pt) node[anchor=south east] {$y_3$}; 

\fill[color=gray, opacity=0.7] (1,0) -- (-0.3,1) -- (-1.1,-0.5);

\draw[thick] (-1.99,2.3) -- (3,-1.5385);
\draw[thick] (0.5,2.5) -- (-1.7933,-1.8);
\draw[thick] (-3.5,-1.0714) -- (3,0.4762);

\end{tikzpicture}}}
     \hfill
     \subfloat[\raggedright the areas covered by \cref{thm:phase1 negative}]{\label{fig:phase1 negative}\resizebox{0.3\linewidth}{!}{\begin{tikzpicture}
\filldraw[black] (1,0) circle (2pt) node[anchor=north] {$y_1$}; 
\filldraw[black] (-0.3,1) circle (2pt) node[anchor=east] {$y_2$}; 
\filldraw[black] (-1.1,-0.5) circle (2pt) node[anchor=south east] {$y_3$}; 

\shade[shading angle=80] (3,0.4762) -- (1,0) -- (3,-1.5385);
\shade[shading angle=-175] (-1.99,2.3) -- (-0.3,1) -- (0.5,2.5);
\shade[shading angle=-40] (-3.5,-1.0714) -- (-1.1,-0.5) -- (-1.7933,-1.8);

\draw[thick] (-1.99,2.3) -- (3,-1.5385);
\draw[thick] (0.5,2.5) -- (-1.7933,-1.8);
\draw[thick] (-3.5,-1.0714) -- (3,0.4762);

\end{tikzpicture}}}
     \hfill
     \subfloat[the areas where \eqref{eq:phase1} is not proved to be sharp]{\label{fig:phase1 not covered} \resizebox{0.3\linewidth}{!}{\begin{tikzpicture}
\shade[shading angle=135] (0.5,2.5) -- (-0.3,1) -- (1,0) -- (3,0.4762);
\shade[shading angle=-115] (-1.99,2.3) -- (-0.3,1) -- (-1.1,-0.5) -- (-3.5,-1.0714);
\shade[shading angle=7] (-1.7933,-1.8) -- (-1.1,-0.5) -- (1,0) -- (3,-1.5385);

\filldraw[black] (1,0) circle (2pt) node[anchor=north] {$y_1$}; 
\filldraw[black] (-0.3,1) circle (2pt) node[anchor=east] {$y_2$}; 
\filldraw[black] (-1.1,-0.5) circle (2pt) node[anchor=south east] {$y_3$}; 

\draw[thick] (-1.99,2.3) -- (3,-1.5385);
\draw[thick] (0.5,2.5) -- (-1.7933,-1.8);
\draw[thick] (-3.5,-1.0714) -- (3,0.4762);

\end{tikzpicture}}}
    \caption{A visualization of results in \cref{sec:phase1} for bivariate interpolation and extrapolation.}
    \label{fig:phase1}
\end{figure}
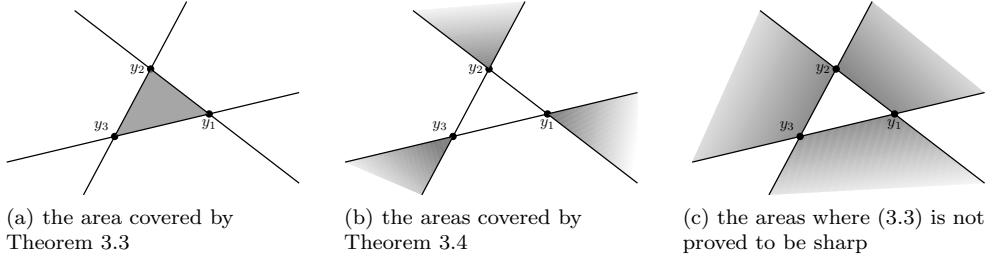

\section{The Worst Quadratic Function} \label{sec:phase2}
Considering under the settings in both \cref{thm:phase1 convex hull} and \cref{thm:phase1 negative} the optimal set of \cref{prob:D} contains at least one quadratic function, so it would be interesting to investigate \cref{prob:D} under the condition that $f$ is quadratic. 
Let $f$ be a quadratic function of the form $f(u) = c + \langle g, u \rangle + \langle Hu, u\rangle / 2$ with $c\in\R, g\in\R^n$, and symmetric $H \in \R^{n \times n}$. 
Because of \cref{eq:Lipschitz Hessian} and  
\[ \begin{aligned}
m(y_0) - f(y_0) ~
&\leftstackrel{\cref{eq:Lagrange m}}{=} \sum_{i=0}^{n+1} \ell_i f(y_i) 
= \sum_{i=0}^{n+1} \ell_i \left[c + \langle g, y_i \rangle + \frac{1}{2} \langle Hy_i, y_i\rangle \right] \\
&\leftstackrel{\cref{eq:Lagrange Y}}{=} \sum_{i=0}^{n+1} \ell_i \left[c + \frac{1}{2} \langle Hy_i, y_i \rangle \right] 
\stackrel{\cref{eq:Lagrange 0}}{=} \sum_{i=0}^{n+1} \ell_i \left[\frac{1}{2} \langle Hy_i, y_i \rangle \right] \\
&= \frac{1}{2}  \langle \sum_{i=0}^{n+1} \ell_i y_i y_i^T, H \rangle 
\stackrel{\cref{eq:G}}{=} \frac{1}{2}  \langle G, H \rangle,  
\end{aligned} \]
the optimization problem described above can be formulated as 
\begin{equation} \label{prob:quadratic} \everymath{\displaystyle} \begin{array}{ll}
    \max_H &\langle G, H \rangle / 2  \\
    \text{s.t.} &-L I \preceq H \preceq L I. 
\end{array} \end{equation}
The absolute sign in the objective function is dropped due to the symmetry of \cref{eq:Lipschitz}, that is, $-f \in C_L^{1,1}(\R^n)$ for any $f \in C_L^{1,1}(\R^n)$, and the objective function value for the two functions $f$ and $-f$ are negatives of each other. 

Problem \cref{prob:quadratic} can be solved analytically. 
Since $G$ is real and symmetric, it must have eigendecomposition $G = P \Lambda P^T$, where $\Lambda \in \R^{n \times n}$ is a diagonal matrix of eigenvalues, and $P \in \R^{n \times n}$ is the orthonormal matrix whose columns are the corresponding eigenvectors. 
The objective function $\langle G,H \rangle/2 = \langle P\Lambda P^T, H \rangle/2 = \langle \Lambda, P^T H P \rangle/2$. 
Since $P$ is orthonormal, the constraint in \cref{prob:quadratic} is equivalent to $-L I \preceq P^T H P \preceq L I$, indicating all diagonal elements of $P^T H P$ are bounded between $-L$ and $L$. 
Since $\Lambda$ is diagonal, only the diagonal elements of $P^T H P$ would affect the objective function value. 
Therefore one solution to \cref{prob:quadratic}, denoted by $H^\star$, has the property $P^T H^\star P = L \text{sign}(\Lambda)$. 
This optimal solution is 
\begin{equation} \label{eq:Hstar}
    H^\star = L P \text{sign}(\Lambda) P^T. 
\end{equation} 

The solution \cref{eq:Hstar} indicates a maximum approximation error of $\langle G, H^\star \rangle/2$ by quadratic functions. 
We next investigate when $\langle G, H^\star \rangle/2$ is an upper bound on the approximation error achievable by all functions in $C^{1,1}_L(\R^n)$. 
To this end, we use a set of parameters $\{\mu_{ij}\}$ and a non-positive function $\psi$. 
The parameters $\{\mu_{ij}\}$ is derived from the eigendecomposition of $G$. 
We first determine in \cref{lem:Sylvester} the number of positive and negative eigenvalues in $G$. 

\begin{lemma} \label{lem:Sylvester}
The numbers of positive and negative eigenvalues in $G$ are $|\cI_+| - 1$ and $|\cI_-|-1$, respectively. 
\end{lemma}
\begin{proof}
Consider the matrix $\bar{G} = \sum_{i=1}^{n+1} \ell_i \phi(y_i-y_0) \phi(y_i-y_0)^T = \Phi^T \text{diag}(\ell) \Phi$. 
The first element of the first column of $\bar{G}$ is $\sum_{i=1}^{n+1} \ell_i \stackrel{\cref{eq:Lagrange 0}}{=} 1$, while the rest of the column is $\sum_{i=1}^{n+1} \ell_i (y_i-y_0) \stackrel{\cref{eq:Lagrange Y}}{=} y_0 - \sum_{i=1}^{n+1} \ell_i y_0 \stackrel{\cref{eq:Lagrange 0}}{=} \mathbf{0}$. 
The bottom-right $n\times n$ submatrix of $\bar{G}$ is  
\[  \sum_{i=1}^{n+1} \ell_i (y_i-y_0) (y_i-y_0)^T 
    = \sum_{i=0}^{n+1} \ell_i (y_i-y_0) (y_i-y_0)^T  
    \stackrel{\cref{eq:G recenter}}{=} G. 
\]
Thus $\bar{G}$ and its eigendecomposition should be 
\[ \bar{G} = \begin{bmatrix} 1 &\mathbf{0}^T\\ \mathbf{0} &G \end{bmatrix}
= \begin{bmatrix} 1 &\mathbf{0}^T\\ \mathbf{0} &P \end{bmatrix}
\begin{bmatrix} 1 &\mathbf{0}^T\\ \mathbf{0} &\Lambda \end{bmatrix}
\begin{bmatrix} 1 &\mathbf{0}^T\\ \mathbf{0} &P^T \end{bmatrix}. 
\]
Then we have 
\[ \bar\Lambda \stackrel{\rm def}{=} \begin{bmatrix} 1 &\mathbf{0}^T\\ \mathbf{0} &\Lambda \end{bmatrix}
= \begin{bmatrix} 1 &\mathbf{0}^T\\ \mathbf{0} &P^T \end{bmatrix} 
\Phi^T \text{diag}(\ell) \Phi
\begin{bmatrix} 1 &\mathbf{0}^T\\ \mathbf{0} &P \end{bmatrix}, 
\]
which shows $\bar\Lambda$ is congruent to $\text{diag}(\ell)$. 
Then by Sylvester's law of inertia \cite{sylvester1852xix} (or \cite{horn2012matrix} Theorem 4.5.8), the number of positive and negative eigenvalues in $\bar\Lambda$ are $|\cI_+|$ and $|\cI_-|-1$, respectively. 
Since $\bar{G}$ shares the same eigenvalues as $G$ except an additional one that is 1, the lemma is proven. 
\end{proof}

Now we partition $\ell, G$, and $H^\star$ with respect to $\cI_+$ and $\cI_-$. 
Remember $\cY$ is assumed to be ordered in a way so that $\ell_1 \ge \ell_2 \ge \cdots \ge \ell_{n+1}$. 
Let $\ell_+ \in \R^{|\cI_+|}$ be the first $|\cI_+|$ elements of $\ell$, and $\ell_- \in \R^{|\cI_-|-1}$ be the last $|\cI_-|-1$ elements. 
Let $Y_+ \in \R^{|\cI_+| \times n}$ be the first $|\cI_+|$ rows of $Y$, and $Y_- \in \R^{(|\cI_-|-1) \times n}$ the last $|\cI_-|-1$ rows. 
Let $\Lambda_+ \in \R^{(|\cI_+|-1) \times (|\cI_+|-1)}$ and $\Lambda_- \in \R^{(|\cI_-|-1) \times (|\cI_-|-1)}$ respectively be the the diagonal matrices that contain the positive and negative eigenvalues of $G$, and $P_+ \in \R^{n \times (|\cI_+|-1)}$ and $P_- \in \R^{n \times (|\cI_-|-1)}$ their corresponding eigenvector matrices. 
Then we have 
\begin{equation} \label{eq:G+-} \begin{aligned} 
G~ &\leftstackrel{\cref{eq:G recenter}}{=} Y^T \diag(\ell) Y  = Y_+^T \diag(\ell_+) Y_+ + Y_-^T \diag(\ell_-) Y_- \\
&= P \Lambda P^T
= P_+ \Lambda_+ P_+^T + P_- \Lambda_- P_-^T
\end{aligned} \end{equation}
and 
\begin{equation} \label{eq:Hstar+-}
    H^\star = L P \text{sign}(\Lambda) P^T
    = L(P_+ P_+^T - P_- P_-^T). 
\end{equation} 

The definition of $\{\mu_{ij}\}$ involves $(Y_- P_-)^{-1}$. 
We prove in \cref{lem:invertible} this inverse is well-defined. 
\begin{lemma} \label{lem:invertible}
The matrix $Y_- P_-$ is invertible. 
\end{lemma}
\begin{proof}
For the purpose of contradiction, assume $Y_- P_-$ is singular. 
That means there is a non-zero vector $u \in \R^{|\cI_-|-1}$ such that $Y_- P_- u = 0$. 
Let $v = P_- u$. 
We have $Y_- v = 0$, $P_+ v = P_+ P_- u = 0$ and $P_-^T v = P_-^T P_- u = u$. 
Then we have the contradiction
\[ \begin{aligned} 
v^T G v &= (Y_+v)^T \text{diag}(\ell_+) Y_+v + (Y_-v)^T \text{diag}(\ell_-) Y_-v = (Y_+v)^T \text{diag}(\ell_+) Y_+v \ge 0 \\
v^T G v &= (P_+^T v)^T \Lambda_+ P_+^T v + (P_-^T v)^T \Lambda_- P_-^T v = (P_-^T v)^T \Lambda_- P_-^T v  = u^T \Lambda_- u < 0.  
\end{aligned} \]
\end{proof}

Now we formally define the parameters $\{\mu_{ij}\}$ and prove its essential properties \cref{eq:mu0 +,eq:mu0 -,eq:mu1 +,eq:mu1 -}. 
\begin{lemma} \label{lem:mu}
Consider the matrix $M \stackrel{\rm def}{=} \text{diag}(\ell_+) Y_+ P_- (Y_- P_-)^{-1}$ and set of real numbers $\{\mu_{ij}\}_{i\in\cI_+,~ j \in\cI_-}$ such that $\mu_{ij} = [M]_{i(j - n-2+|\cI_-|)}$ for all $(i,j)\in\cI_+\times(\cI_-\setminus\{0\})$ and $\mu_{i0} = \ell_i - \sum_{j \in \cI_-\setminus\{0\}} \mu_{ij}$ for all $i\in\cI_+$. 
The following properties hold: 
\begin{align}
    \sum_{j \in \cI_-} \mu_{ij} &= \ell_i &&\text{for all } i \in \cI_+, \label{eq:mu0 +}\\
    \sum_{i \in \cI_+} \mu_{ij} &= - \ell_j &&\text{for all } j \in \cI_-, \label{eq:mu0 -}\\
    (LI-H^\star) \sum_{j \in \cI_-} \mu_{ij} y_j &= (LI-H^\star) \ell_i y_i &&\text{for all } i \in \cI_+, \label{eq:mu1 +}\\
    (LI+H^\star) \sum_{i \in \cI_+} \mu_{ij}y_i &= -(LI+H^\star) \ell_j y_j &&\text{for all } j \in \cI_-. \label{eq:mu1 -}
\end{align}
\end{lemma}
\begin{proof}
The equations \cref{eq:mu0 +} are true by their definition. 
Since
\[ \begin{aligned}
    \text{diag}(\ell_-) \mathbf{1} + M^T \mathbf{1}
    &= \text{diag}(\ell_-) \mathbf{1} + (P_-^T Y_-^T)^{-1} P_-^T Y_+^T \text{diag}(\ell_+) \mathbf{1} \\
    &= (P_-^T Y_-^T)^{-1} P_-^T [Y_-^T \text{diag}(\ell_-) \mathbf{1} + Y_+^T \text{diag}(\ell_+) \mathbf{1}] 
    \stackrel{\cref{eq:Lagrange Y}}{=} \mathbf{0}, 
\end{aligned} \]
the equations \cref{eq:mu0 -} are also true. 
Notice $P_-^T (Y_-^T M^T - Y_+^T \text{diag}(\ell_+)) = \mathbf{0}$ by the definition of $M$, and $LI - H^\star \stackrel{\cref{eq:Hstar+-}}{=} L(P_+P_+^T + P_-P_-^T) - L(P_+P_+^T - P_-P_-^T) = 2L P_-P_-^T$. 
Following these two equations, we have for all $i \in \cI_+$, 
\[ \begin{aligned}
    (LI-H^\star) \left[\sum_{j \in \cI_-} \mu_{ij} y_j - \ell_i y_i\right] 
    &\leftstackrel{\cref{eq:mu0 -}}{=} (LI-H^\star) \left[\sum_{j \in \cI_-} \mu_{ij} (y_j-y_0) - \ell_i (y_i-y_0)\right] \\
    &= (LI-H^\star) (Y_-^T M^T - Y_+^T \text{diag}(\ell_+)) e_i \\
    &= 2L P_- P_-^T (Y_-^T M^T - Y_+^T \text{diag}(\ell_+)) e_i \\
    &= 2L P_- \mathbf{0} e_i = 0, 
\end{aligned} \] 
which proves \cref{eq:mu1 +}. 
To prove \cref{eq:mu1 -}, we use $G$ and its eigendecomposition. 
The diagonal matrix of the eigenvalues $\Lambda$ is 
\[  \begin{bmatrix} \Lambda_- &\mathbf{0}\\ \mathbf{0} &\Lambda_+ \end{bmatrix} 
    = \begin{bmatrix} P_-^T Y_-^T  &P_-^T Y_+^T\\ P_+^T Y_-^T &P_+^T Y_+^T \end{bmatrix}
    \begin{bmatrix} \text{diag}(\ell_-) &\mathbf{0}\\ \mathbf{0} &\text{diag}(\ell_+) \end{bmatrix}
    \begin{bmatrix} Y_- P_- &Y_- P_+\\ Y_+ P_- &Y_+ P_+ \end{bmatrix}, 
\]
which contains two equivalent block equalities with zero left-hand side. 
They are $P_+^T Y_-^T \diag(\ell_-) Y_- P_- + P_+^T Y_+^T \diag(\ell_+) Y_+ P_- = \mathbf{0}$, so 
\[ P_+^T Y_-^T \diag(\ell_-) + P_+^T Y_+^T \diag(\ell_+) Y_+ P_- (Y_- P_-)^{-1} = P_+^T Y_-^T \diag(\ell_-) + P_+^T Y_+^T M = \mathbf{0}. 
\]
Then with $LI+H^\star \stackrel{\cref{eq:Hstar+-}}{=} L(P_+P_+^T + P_-P_-^T) + L(P_+P_+^T - P_-P_-^T) = 2L P_+ P_+^T$, we obtain
\[  (LI + H^\star) (Y_-^T \diag(\ell_-) + Y_+^T M)
    = 2L P_+ P_+^T (Y_-^T \diag(\ell_-) + Y_+^T M) 
    = 2L P_+ \mathbf{0} = \mathbf{0}, 
\]
which proves \cref{eq:mu1 -} for all $j \in \cI_- \setminus \{0\}$; and 
\[ \begin{aligned} 
    (LI+H^\star) \left( \ell_0 y_0 + \sum_{i\in\cI_+} \mu_{i0} y_i \right) 
    &\leftstackrel{\cref{eq:mu0 -}}{=} 2LP_+P_+^T \sum_{i\in\cI_+} \mu_{i0} (y_i-y_0) \\
    &= 2LP_+P_+^T \sum_{i\in\cI_+} \left(\ell_i - \sum_{j\in\cI_-\setminus\{0\}} \mu_{ij}\right) (y_i-y_0) \\
    &= 2LP_+P_+^T Y_+^T (l_+ - M\mathbf{1}) \\
    &= 2LP_+P_+^T (Y_+^T l_+ + Y_-^T \ell_-)
    \stackrel{\cref{eq:Lagrange Y}}{=} \mathbf{0}, 
\end{aligned} \]
which proves \cref{eq:mu1 -} for $j = 0$. 
\end{proof}

The function $\psi$ is defined and proved non-positive in \cref{lem:psi}. 
It will be used to prove \cref{thm:phase2} in conjunction with the parameters $\{\mu_{ij}\}$. 
\begin{lemma} \label{lem:psi}
Assume $f \in C^{1,1}_L(\R^n)$. 
For any $(u_1, u_2) \in \R^n \times \R^n$ and any matrix $H \in \R^{n \times n}$, we have 
\begin{equation} \label{eq:Lipscthiz stronger H} \begin{aligned} 
    \psi(u_1,u_2,H) \stackrel{\rm def}{=} &f(u_1) -  f(u_2) - \frac{1}{2L} \langle (LI-H)(u_1-u_2), \nabla f(u_1) \rangle \\
    &- \frac{1}{2L} \langle (LI+H) (u_1-u_2), \nabla f(u_2) \rangle \\
    &- \frac{1}{4L} \|H (u_1 - u_2)\|^2 - \frac{L}{4} \|u_1 - u_2\|^2 \le 0. 
\end{aligned} \end{equation} 
\end{lemma}
\begin{proof}
For the purpose of contradiction, assume \cref{eq:Lipscthiz stronger H} is false. Then we have 
\begin{multline*}
    - f(u_1) < - f(u_2) - \frac{1}{2L} \langle (LI+H) (u_1-u_2), \nabla f(u_1) \rangle\\ - \frac{1}{2L} \langle (LI-H)(u_1-u_2), \nabla f(u_2) \rangle - \frac{1}{4L} \|H (u_1 - u_2)\|^2 - \frac{L}{4} \|u_1 - u_2\|^2. 
\end{multline*}
Add this inequality to \cref{eq:Lipschitz stronger} and we arrive at 
\[ \frac{1}{4L} \|H (u_1 - u_2) - (\nabla f(u_1) - \nabla f(u_2))\|^2 < 0, 
\] 
which leads to contradiction. 
\end{proof}

Finally, we prove in \cref{thm:phase2} that $\langle G, H^\star \rangle /2$ is a sharp bound when $\{\mu_{ij}\}$ are all non-negative. 
\begin{theorem} \label{thm:phase2} 
Assume $f \in C^{1,1}_L(\R^n)$. 
Let $m$ be the linear function that interpolates $f$ at any set of $n+1$ affinely independent vectors $\cY = \{y_1,\dots,y_{n+1}\}\subset \R^n$. 
Let $y_0$ be any vector in $\R^n$. 
Let $G$ and $H^\star$ be matrices defined in \cref{eq:G} and \cref{eq:Hstar}. 
Let $\{\mu_{ij}\}_{i \in \cI_+,~ j \in \cI_-}$ be the set of parameters defined in \cref{lem:mu}. 
If $\mu_{ij} \ge 0$ for all $(i,j) \in \cI_+\times\cI_-$, then the function approximation error of $m$ at $y_0$ is bounded as 
\begin{equation} \label{eq:phase2}
    |m(y_0) - f(y_0)| \le \frac{1}{2} \langle G, H^\star \rangle. 
\end{equation}
\end{theorem}
\begin{proof}
We only provide the proof for the case when $m(y_0) - f(y_0) \ge 0$. 
When $\mu_{ij} \ge 0$ for all $(i,j) \in \cI_+\times\cI_-$, the following inequality holds 
\begin{equation} \label{eq:phase2 summation}
    \sum_{i \in \cI_+} \sum_{j \in \cI_-} \mu_{ij} \psi(y_i,y_j,H^\star) \stackrel{\cref{eq:Lipscthiz stronger H}}{\le} 0. 
\end{equation}
The zeroth-order term in the summation \cref{eq:phase2 summation} is 
\[ \begin{aligned}
    \sum_{i \in \cI_+} \sum_{j \in \cI_-} \mu_{ij} (f(y_i) - f(y_j)) 
    &= \left[\sum_{i \in \cI_+} \sum_{j \in \cI_-} \mu_{ij} f(y_i)\right] - \left[\sum_{i \in \cI_+} \sum_{j \in \cI_-} \mu_{ij} f(y_j)\right] \\
    &\leftstackrel{\cref{eq:mu0 +}\cref{eq:mu0 -}}{=} \left[ \sum_{i \in \cI_+} \ell_i f(y_i)\right] + \left[\sum_{j \in \cI_-} \ell_j f(y_j)\right] \\
    &\leftstackrel{\cref{eq:Lagrange m}}{=} m(y_0) - f(y_0). 
\end{aligned} \] 
The sum of the first-order terms is $-1/(2L)$ multiplies
\[ \begin{aligned} 
    &\sum_{i \in \cI_+} \sum_{j\in\cI_-} \mu_{ij} \big( \langle(LI-H^\star)(y_i-y_j), \nabla f(y_i) \rangle + \langle (LI+H^\star)(y_i-y_j), \nabla f(y_j)\rangle \big) \\
    &= \left[\sum_{i \in \cI_+} \sum_{j\in\cI_-} \mu_{ij} \langle(LI-H^\star)y_i, \nabla f(y_i) \rangle \right] 
    - \left[\sum_{i \in \cI_+} \sum_{j\in\cI_-} \mu_{ij} \langle (LI+H^\star)y_j, \nabla f(y_j)\rangle \right] \\
    &\quad -\left[\sum_{i \in \cI_+} \sum_{j\in\cI_-} \mu_{ij}  \langle(LI-H^\star)y_j, \nabla f(y_i) \rangle \right] + \left[\sum_{i \in \cI_+} \sum_{j\in\cI_-} \mu_{ij} \langle(LI+H^\star)y_i, \nabla f(y_j) \rangle \right] \\
    &= \left[\sum_{i \in \cI_+} \ell_i \langle(LI-H^\star)y_i, \nabla f(y_i) \rangle \right] 
    + \left[\sum_{j\in\cI_-} \ell_j \langle (LI+H^\star)y_j, \nabla f(y_j)\rangle \right] \\
    &\quad -\left[\sum_{i \in \cI_+} \ell_i \langle(LI-H^\star)y_i, \nabla f(y_i) \rangle \right] - \left[\sum_{j\in\cI_-} \ell_j \langle(LI+H^\star)y_j, \nabla f(y_j) \rangle \right] 
    = \mathbf{0}, 
\end{aligned} \]
where the second equality holds because of \cref{eq:mu0 +}, \cref{eq:mu0 -}, \cref{eq:mu1 +}, and \cref{eq:mu1 -} respectively for the four terms. 
Notice $H^{\star T} H^\star = L^2I$. 
The constant term in the summation \cref{eq:phase2 summation} is $-1/2$ multiplies
\[ \begin{aligned} 
    &\sum_{i \in \cI_+} \sum_{j\in\cI_-} \mu_{ij} \left(\frac{1}{2L}\|H^\star(y_i-y_j)\|^2 + \frac{L}{2} \|y_i-y_j\|^2\right) \\
    &= L \left[ \sum_{i \in \cI_+} \sum_{j\in\cI_-} \mu_{ij} \langle y_i-y_j, y_i\rangle \right] - L \left[ \sum_{i \in \cI_+} \sum_{j\in\cI_-} \mu_{ij} \langle y_i-y_j, y_j\rangle \right] \\
    &\stackrel{\mathmakebox[\widthof{=}]{\scriptsize \begin{array}{c}\cref{eq:mu0 +}\\\cref{eq:mu0 -}\end{array}}}{=} \left[ \sum_{i \in \cI_+} \langle L \left(\ell_i y_i- \sum_{j\in\cI_-} \mu_{ij} y_j\right), y_i\rangle \right] - \left[ \sum_{j\in\cI_-} \langle L \left(\sum_{i \in \cI_+} \mu_{ij} y_i + \ell_j y_j\right), y_j\rangle \right] \\
    &\leftstackrel{\mathmakebox[\widthof{=}]{\scriptsize \begin{array}{c}\cref{eq:mu1 +}\\ \cref{eq:mu1 -}\end{array}}}{=} \left[ \sum_{i \in \cI_+} \langle H^\star \left(\ell_i y_i- \sum_{j\in\cI_-} \mu_{ij} y_j\right), y_i\rangle \right] + \left[ \sum_{j\in\cI_-} \langle H^\star \left(\sum_{i \in \cI_+} \mu_{ij} y_i + \ell_j y_j\right), y_j\rangle \right] \\
    &= \left[ \sum_{i \in \cI_+} \ell_i \langle H^\star y_i, y_i\rangle \right] + \left[ \sum_{j\in\cI_-} \ell_j \langle H^\star y_j y_j\rangle \right] 
    = \langle G, H^\star \rangle. 
\end{aligned} \]
Thus the summation \cref{eq:phase2 summation} is \cref{eq:phase2} when $m(y_0) - f(y_0) \ge 0$. 
\end{proof}

\cref{thm:phase2} proved that \cref{eq:phase2} is a sharp bound under the condition $\mu_{ij} \ge 0$ for all $(i,j) \in \cI_+ \times \cI_-$, but the geometric meaning of this condition is obscure. 
We numerically generated many different $\cY$ and $y_0$ with various $n$ and calculated the parameters $\{\mu_{ij}\}$. 
Based on our observation, we believe the following statements are true. 
\begin{enumerate}
    \item When there is no obtuse angle at the vertices of the simplex $\conv(\cY)$, that is, when 
    \begin{equation} \label{eq:acute simplex} 
        \langle y_j-y_i, y_k-y_i \rangle \ge 0 \text{ for all } (i,j,k) \in \{1,2,\dots,n+1\}^3, 
    \end{equation}
    the parameters $\{\mu_{ij}\}_{i\in\cI_+, j\in\cI_-}$ are all non-negative for any $y_0 \in \R^n$ and \cref{eq:phase2} is therefore a sharp bound. 
    \item If there is at least one obtuse angle at the vertices of the simplex $\conv(\cY)$, then there is a non-empty subset of $\R^n$ to which if $y_0$ belongs, there is at least one negative element in $\{\mu_{ij}\}_{i\in\cI_+, j\in\cI_-}$ and the bound \cref{eq:phase2} is invalid as a function $f \in C_L^{1,1}(\R^n)$ with a larger approximation error exists. 
    \item For bivariate interpolation and extrapolation, the bound \cref{eq:phase2} is invalid if and only if the simplex formed by $\cY$ is an obtuse triangle and $y_0$ is inside what is indicated by the four shaded areas in \cref{fig:phase2}. 
    These shaded areas are open subsets of $\R^2$ and do not include their boundaries. 
\end{enumerate}
We cannot prove these statements mathematically, but we can partly show the third statement is true by deriving in \cref{sec:phase3} the sharp error bound when $y_0$ is in shaded areas in \cref{fig:phase2}.

\begin{figure}[tbhp]
  \centering
    \resizebox{0.4\textwidth}{!}{\begin{tikzpicture}
\filldraw[black] (0,0) circle (2pt) node[anchor=north west] {$y_1$}; 
\filldraw[black] (2,1.8) circle (2pt) node[anchor=south east] {$y_2$}; 
\filldraw[black] (-2,0) circle (2pt) node[anchor=south] {$y_3$}; 

\fill[color=gray, opacity=0.7] (0,0) -- (2,1.8) -- (2,0); 
\fill[color=gray, opacity=0.7] (0,0) -- (-2,0) -- (-1.1050, -0.9945); 
\shade[left color=gray, right color=white, shading angle=156] (2,3.5) -- (2,1.8) -- (3.333,3);
\shade[left color=gray, right color=white, shading angle=-120] (-3.7,0) -- (-2,0) -- (-3.35,1.5);

\draw[thick] (-2,-1.8) -- (3.333,3);
\draw[thick] (-3,-0.45) -- (4,2.7);
\draw[thick] (-3.7,0) -- (3.333,0);
\draw[thick, dashed] (2,-0.5) -- (2,3.5); 
\draw[thick, dashed] (-3.35,1.5) -- (-0.2,-2); 

\end{tikzpicture}}
  \caption{The areas to which if $y_0$ belongs, \cref{eq:phase2} is invalid for bivariate extrapolation. 
  The dashed line on the left is perpendicular to the line going through $y_1$ and $y_2$; and the one on the right is perpendicular to the line going through $y_3$ and $y_1$. } 
  \label{fig:phase2}
\end{figure}
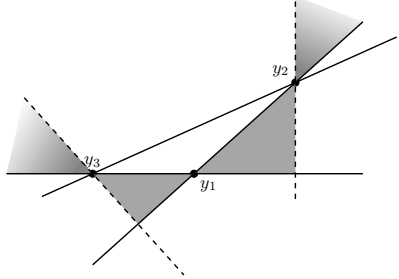

\section{Sharp Error Bounds for Bivariate Extrapolation} \label{sec:phase3}
Because of the symmetries between the two shaded triangles and the two shaded cones in \cref{fig:phase2}, we only need to derive the error bounds for the two cases indicated in \cref{fig:phase3}. 
We first investigate the case in \cref{fig:phase3 triangle}, which can be defined mathematically under barycentric coordinate system as $\ell_2>0, \ell_3<0$, and $\ell_1\langle y_2-y_1,y_3-y_1\rangle - \ell_3\langle y_2-y_3, y_1-y_3\rangle < 0$, and then argue that the case in \cref{fig:phase3 cone} is analogous to the case in \cref{fig:phase3 triangle}.
The following lemma shows the point $w$, as defined in \cref{eq:phase3 w}, is the intersection of the line going through $y_1$ and $y_3$ and the line going through $y_0$ and $y_2$. 

\begin{figure}[tbhp]
     \centering
     \subfloat[\raggedright when $y_0$ is in the open triangle such that $\ell_2>0, \ell_3<0$, and $\ell_1\langle y_2-y_1,y_3-y_1\rangle - \ell_3\langle y_2-y_3, y_1-y_3\rangle < 0$]{\label{fig:phase3 triangle}\resizebox{0.4\linewidth}{!}{\begin{tikzpicture}
\filldraw[black] (0,0) circle (2pt) node[anchor=north east] {$y_1$}; 
\filldraw[black] (2,1.8) circle (2pt) node[anchor=south east] {$y_2$}; 
\filldraw[black] (-2,0) circle (2pt) node[anchor=south] {$y_3$};  
\filldraw[black] (0.9,0) circle (2pt) node[anchor=north west] {$w$}; 
\filldraw[black] (1.3889,0.8) circle (2pt) node[anchor=north west] {$y_0$}; 

\draw[thick] (0,0) -- (3.333,3);
\draw[thick] (2,1.8) -- (-2,0);
\draw[thick] (-2,0) -- (3.333,0);
\draw[thick, dashed] (2,-0.5) -- (2,3.5);
\draw[thick, dashed] (2,1.8) -- (0.9,0);

\end{tikzpicture}}}
     \hspace{0.1\linewidth}
     \subfloat[\raggedright when $y_0$ is in the open cone such that $l_3>0$ and $\ell_1\langle y_2-y_1,y_3-y_1\rangle - \ell_3\langle y_2-y_3, y_1-y_3\rangle > 0$]{\label{fig:phase3 cone}\resizebox{0.4\linewidth}{!}{ \begin{tikzpicture}
\filldraw[black] (0,0) circle (2pt) node[anchor=north east] {$y_1$}; 
\filldraw[black] (2,1.8) circle (2pt) node[anchor=south east] {$y_2$}; 
\filldraw[black] (-2,0) circle (2pt) node[anchor=south] {$y_3$};  
\filldraw[black] (0.9,0) circle (2pt) node[anchor=north west] {$w$}; 
\filldraw[black] (2.4278,2.5) circle (2pt) node[anchor=south west] {$y_0$}; 

\draw[thick] (0,0) -- (3.333,3);
\draw[thick] (2,1.8) -- (-2,0);
\draw[thick] (-2,0) -- (3.333,0);
\draw[thick, dashed] (2,-0.5) -- (2,3.5);
\draw[thick, dashed] (2.4278,2.5) -- (0.9,0);

\end{tikzpicture}}}
    \caption{Two configurations of $\cY$ and $y_0$ where \cref{eq:phase2} is invalid for bivariate extrapolation.}
    \label{fig:phase3}
\end{figure}
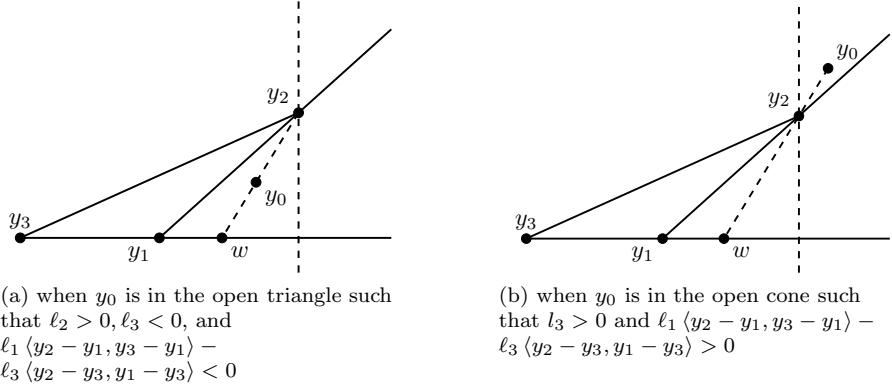

\begin{lemma}
Assume $-\ell_0-\ell_2 \stackrel{\cref{eq:Lagrange 0}}{=} \ell_1+\ell_3\neq 0$ for some affinely independent $\cY\subset\R^2$ and $y_0\in\R^2$. 
Let 
\begin{equation} \label{eq:phase3 w}
    w = \frac{-\ell_0y_0+\ell_1y_1-\ell_2y_2+\ell_3y_3}{-\ell_0+\ell_1-\ell_2+\ell_3}. 
\end{equation}
Then 
\[ w = \frac{\ell_1y_1+\ell_3y_3}{\ell_1+\ell_3} = \frac{\ell_0y_0+\ell_2y_2}{\ell_0+\ell_2}, 
\]
and 
\begin{subequations} \label{eq:phase3 w y} \begin{align}
    \ell_0(y_0-w) + \ell_2(y_2-w) &= 0, \\
    \ell_1(y_1-w) + \ell_3(y_3-w) &= 0. 
\end{align} \end{subequations}
\end{lemma}
\begin{proof}
These equalities are direct results of \cref{eq:Lagrange 0,eq:Lagrange Y}. 
\end{proof}

We define in the following lemma an $H^\star$, which is different from the one defined in \cref{eq:Hstar} and is asymmetric. 
\begin{lemma}
Assume for some affinely independent $\cY\subset\R^2$ and $y_0\in\R^2$ that $\ell_2>0, \ell_3<0$, and $\ell_1\langle y_2-y_1,y_3-y_1\rangle - \ell_3\langle y_2-y_3, y_1-y_3\rangle < 0$. 
Let 
\begin{equation} \label{eq:Hstar 3}
    H^\star = P \begin{bmatrix} +L &0\\ 0 &-L \end{bmatrix} P^{-1} \text{ with } P = \begin{bmatrix} y_2-y_0 &y_1-y_3 \end{bmatrix}. 
\end{equation}
Let $w$ be defined as \cref{eq:phase3 w}. 
Then 
\begin{equation} \label{eq:phase3 Hstar eigvector} \begin{aligned} 
    H^\star(y_i-w) &= L(y_i-w) \text{ for } i\in\{0,2\}, \\
    H^\star(y_i-w) &= -L(y_i-w) \text{ for } i\in\{1,3\}. 
\end{aligned} \end{equation} 
\end{lemma}
\begin{proof}
It is clear from \cref{fig:phase3 triangle} that the assumption guarantees the invertibility of $P$ and $-\ell_0-\ell_2 = \ell_1+\ell_3\neq 0$. 
Notice by the definition of $H^\star$, we have $H^\star(y_2-y_0) = L(y_2-y_0)$ and $H^\star(y_1-y_3) = -L(y_1-y_3)$.
The lemma holds true because $y_i-w$ is parallel to $y_2-y_0$ for $i\in\{0,2\}$ and to $y_1-y_3$ for $i\in\{1,3\}$. 
\end{proof}

Now we are ready to show $\langle G, H^\star \rangle/2$, with $H^\star$ defined in \cref{eq:Hstar 3}, is an upper bound on the function approximation error in the case in \cref{fig:phase3 triangle}. 
\begin{theorem} \label{thm:phase3}
Assume $f \in C^{1,1}_L(\R^2)$. 
Let $m$ be the linear function that interpolates $f$ at any set of three affinely independent vectors $\cY = \{y_1,y_2,y_3\}\subset \R^2$ such that $\langle y_2-y_1, y_3-y_1 \rangle < 0$. 
Let $y_0$ be any vector in $\R^2$ such that its barycentric coordinates satisfies $\ell_2>0, \ell_3<0$, and $\ell_1\langle y_2-y_1,y_3-y_1\rangle - \ell_3\langle y_2-y_3, y_1-y_3\rangle < 0$. 
Let $G$ and $H^\star$ be the matrices defined in \cref{eq:G} and \cref{eq:Hstar 3}. 
Then the function approximation error of $m$ at $y_0$ is bounded as 
\begin{equation} \label{eq:phase3}
    |m(y_0) - f(y_0)| \le \frac{1}{2} \langle G, H^\star \rangle. 
\end{equation}
\end{theorem}

\begin{proof}
We only provide the proof for the case when $m(y_0)-f(y_0) \ge 0$. 
We use the function $\psi$ defined in \cref{eq:Lipscthiz stronger H} again. 
Since $\ell_3<0$, $\langle y_2-y_1, y_3-y_1 \rangle < 0$, and 
\[ \begin{aligned}
    0 &> \ell_1\langle y_2-y_1,y_3-y_1\rangle - \ell_3\langle y_2-y_3, y_1-y_3\rangle \\
    &\leftstackrel{\cref{eq:Lagrange 0}}{=}  (1-\ell_2-\ell_3) \langle y_2-y_1,y_3-y_1\rangle - \ell_3\langle y_2-y_3, y_1-y_3\rangle \\
    &= (1-\ell_2) \langle y_2-y_1,y_3-y_1\rangle - \ell_3 \|y_1-y_3\|^2, 
\end{aligned} \] 
we have $1-\ell_2 > 0$, and thus the following inequalities hold: 
\begin{subequations} \label{eq:phase3 sum} \begin{align}
    (1-\ell_2) \psi(y_1, y_0, H^\star) &\le 0, \\
    \ell_2 \psi(y_2, y_0, H^\star) &\le 0, \\ 
    -\ell_3 \psi(y_1, y_3, H^\star) &\le 0. 
\end{align} \end{subequations}
Similar to the previous proofs, we add these inequalities together. 
The sum of their zeroth-order terms is 
\[ \begin{aligned} 
&\hspace{-1em} (1-\ell_2) (f(y_1) - f(y_0)) + \ell_2 (f(y_2) - f(y_0)) - \ell_3 (f(y_1) - f(y_3)) \\
&= (1-\ell_2-\ell_3) f(y_1) + \ell_2 f(y_2) + \ell_3 f(y_3) - f(y_0) 
\stackrel{\cref{eq:Lagrange m}\cref{eq:Lagrange 0}}{=} m(y_0) - f(y_0). 
\end{aligned} \]
The sum of their first-order terms is $-1/(2L)$ multiplies 
\[ \begin{aligned} 
&\hspace{-2em} (1-\ell_2) \left[ \langle (LI-H^\star) (y_1-y_0), \nabla f(y_1) \rangle + \langle (LI+H^\star) (y_1-y_0), \nabla f(y_0) \rangle \right] \\
&\hspace{-2em} + \ell_2 \left[ \langle (LI-H^\star) (y_2-y_0), \nabla f(y_2) \rangle + \langle (LI+H^\star) (y_2-y_0), \nabla f(y_0) \rangle \right] \\
&\hspace{-2em} -\ell_3 \left[ \langle (LI-H^\star) (y_1-y_3), \nabla f(y_1) \rangle + \langle (LI+H^\star) (y_1-y_3), \nabla f(y_3) \rangle \right] \\
&= \langle (LI-H^\star) [(1-\ell_2)(y_1-y_0) - \ell_3(y_1-y_3)], \nabla f(y_1) \rangle \\
&\quad + \ell_2 \langle (LI-H^\star) (y_2-y_0), \nabla f(y_2) \rangle \\
&\quad - \ell_3 \langle (LI+H^\star) (y_1-y_3), \nabla f(y_3) \rangle \\
&\quad + \langle (LI+H^\star) [(1-\ell_2)(y_1-y_0) + \ell_2(y_2-y_0)], \nabla f(y_0) \rangle \\ 
&\leftstackrel{\cref{eq:Lagrange 0}\cref{eq:Lagrange Y}}{=} \ell_2 \langle (LI-H^\star) (y_0 - y_2), \nabla f(y_1) \rangle + \ell_2 \langle (LI-H^\star) (y_2-y_0), \nabla f(y_2) \rangle \\
&\quad - \ell_3 \langle (LI+H^\star) (y_1-y_3), \nabla f(y_3) \rangle + \ell_3 \langle (LI+H^\star) (y_1 - y_3), \nabla f(y_0) \rangle \\ 
&\leftstackrel{\cref{eq:phase3 Hstar eigvector}}{=} \mathbf{0}. 
\end{aligned} \]
Let $w$ be defined as \cref{eq:phase3 w}. 
The sum of the constant terms is $-1/2$ times 
\[ \begin{aligned} 
&\hspace{-1em} (1-\ell_2) \left(\frac{1}{2L} \|H^\star (y_1-y_0)\|^2 + \frac{L}{2} \|y_1-y_0\|^2 \right) + \ell_2 \left(\frac{1}{2L} \|H^\star (y_2-y_0)\|^2 \right. \\
&\hspace{-1em} \left.+ \frac{L}{2} \|y_2-y_0\|^2 \right) - \ell_3 \left(\frac{1}{2L} \|H^\star (y_1-y_3)\|^2 + \frac{L}{2} \|y_1-y_3\|^2 \right) \\
&\leftstackrel{\cref{eq:phase3 Hstar eigvector}}{=} (1-\ell_2)\left( -\langle H^\star(y_1-w), y_1-w \rangle + \langle H^\star(y_0-w), y_0-w \rangle \right) \\
&\qquad + \ell_2 \langle H^\star(y_2-y_0), y_2-y_0 \rangle + \ell_3 \langle H^\star(y_1-y_3), y_1-y_3 \rangle \\
&\leftstackrel{\cref{eq:Lagrange 0}}{=} \langle H^\star[\ell_3(y_1-y_3)-(\ell_1+\ell_3)(y_1-w)], y_1-w \rangle - \ell_3 \langle H^\star(y_1-y_3), y_3-w \rangle \\
&\qquad + \langle H^\star[(1-\ell_2)(y_0-w) - \ell_2(y_2-y_0)], y_0-w \rangle + \ell_2 \langle H^\star(y_2-y_0), y_2-w \rangle \\
&\leftstackrel{\cref{eq:Lagrange 0}\cref{eq:Lagrange Y}}{=} 0 - \ell_3 \langle H^\star(y_1-w) - H^\star(y_3-w), y_3-w \rangle \\
&\quad + 0 + \ell_2 \langle H^\star(y_2-w) - H^\star(y_0-w), y_2-w \rangle \\ 
&\leftstackrel{\cref{eq:phase3 w y}}{=} \sum_{i=0}^{3} \ell_i \langle H^\star (y_i-w), y_i-w \rangle 
\stackrel{\cref{eq:G recenter}}{=} \langle G, H^\star \rangle. 
\end{aligned} \]
Thus the sum of the inequalities in \cref{eq:phase3 sum} is \cref{eq:phase3} when $m(y_0) - f(y_0) \ge 0$. 
\end{proof} 

The sharpness of \cref{eq:phase3} is proved in \cref{thm:phase3 sharp}. 
\begin{theorem} \label{thm:phase3 sharp} 
Under the setting of \cref{thm:phase3}, the bound \cref{eq:phase3} is sharp. 
\end{theorem}
\begin{proof}
Let $w$ be defined as \cref{eq:phase3 w} and consider the piecewise quadratic function 
\[ f(u) = \left\{ \begin{aligned}
    &\frac{L}{2} \|u-w\|^2 - L\langle \frac{y_1-y_3}{\|y_1-y_3\|}, u-w \rangle^2 &&\text{if } \langle u-w, y_1-y_3 \rangle \le 0,  \\
    &\frac{L}{2} \|u-w\|^2 &&\text{if } \langle u-w, y_1-y_3 \rangle \ge 0. 
\end{aligned} \right. 
\]
Its function approximation error is 
\[ \begin{aligned}
m(&y_0) - f(y_0)
= \sum_{i=0}^{n+1} \ell_i y_i \\
&= \frac{L}{2} \sum_{i=0}^3 \ell_i \|y_i-w\|^2 - L\ell_1\langle \frac{y_1-y_3}{\|y_1-y_3\|}, y_1-w \rangle^2 - 2L\ell_3\langle \frac{y_1-y_3}{\|y_1-y_3\|}, y_3-w \rangle^2 \\
&= \frac{L}{2} \sum_{i=0}^3 \ell_i \|y_i-w\|^2 - L\ell_1\|y_1-w\|^2 - 2L\ell_3\|y_3-w\|^2 \\
&= \frac{L}{2} \left(\ell_0\|y_0-w\|^2 - \ell_1\|y_1-w\|^2 + \ell_2\|y_2-w\|^2 - \ell_3\|y_3-w\|^2\right) \\
&\leftstackrel{\cref{eq:phase3 Hstar eigvector}}{=} \frac{1}{2} \sum_{i=0}^3 \ell_i \langle y_i-w, H^\star(y_i-w) \rangle 
\stackrel{\cref{eq:G recenter}}{=} \frac{1}{2} \langle G, H^\star \rangle. 
\end{aligned} \]
Now we prove $f \in C_L^{1,1}(\R^n)$. 
Firstly, it is clear that $f$ is continuous on $\R^2$ and differentiable on the two half spaces $\{u:~ \langle u-w, y_1-y_3 \rangle < 0\}$ and $\{u:~ \langle u-w, y_1-y_3 \rangle > 0\}$. 
Then given any $u$ such that $\langle u-w, y_1-y_3 \rangle = 0$, it can be calculated for any $v\in\R^2$ that 
\begin{multline*}
    |f(u+v) - f(u) - \langle L(u-w), v \rangle| \\
    = \left\{ \begin{aligned} 
    &-\frac{L}{2}\|v\|^2 - L \langle \frac{y_1-y_3}{\|y_1-y_3\|}, v\rangle^2 &&\text{if }  \langle u+v-w, y_1-y_3 \rangle \le 0,  \\
    &-\frac{L}{2}\|v\|^2 &&\text{if }  \langle u+v-w, y_1-y_3 \rangle \ge 0. 
    \end{aligned} \right.
\end{multline*} 
Thus 
\[ \lim_{v\rightarrow\mathbf{0}} \frac{|f(u+v) - f(u) - \langle L(u-w), v \rangle|}{\|v\|} = 0, 
\]
which shows $f$ is differentiable with gradient $L(u-w)$ on $\{u:~ \langle u-w, y_1-y_3 \rangle = 0\}$. 
The condition \cref{eq:Lipschitz} is clearly satisfied if $u_1$ and $u_2$ are in the same half space. 
Now assume $\langle u_1-w, y_1-y_3 \rangle < 0$ and $\langle u_2-w, y_1-y_3 \rangle > 0$. 
Then, we have 
\[ \begin{aligned} 
&\|\nabla f(u_1) - \nabla f(u_2)\|^2 \\
&= \|L(u_1-w) - 2L\langle y_1-y_3, u_1-w\rangle (y_1-y_3)/\|y_1-y_3\|^2 - L(u_2-w)\|^2 \\
&= L^2\|u_1-u_2\|^2 + 4L^2 \langle u_1-w, y_1-y_3\rangle \langle u_2-w, y_1-y_3 \rangle /\|y_1-y_3\|^2 \\
&< L^2\|u_1-u_2\|^2, 
\end{aligned} \]
which shows \cref{eq:Lipschitz} always holds.
Therefore $f \in C_L^{1,1}(\R^n)$. 
\end{proof}

Now consider problem \cref{prob:D} in the case depicted in \cref{fig:phase3 cone}. 
If we divide the objective function by $-\ell_2<0$, the coefficient before $f(y_i)$ becomes $\alpha_i = -\ell_i/\ell_2$ for all $i\in\{1,2,\dots,n+1\}$. 
We have $\sum_{i=0}^{n+1} \alpha_i = 0$ and $\sum_{i=0}^{n+1} \alpha_i y_i = 0$. 
Moreover, we have $\alpha_2 = -1$, so we can treat the new optimization problem as if the objective function is $m(y_2)-f(y_2)$ with $m$ being the linear function interpolating $f$ on $\{y_0,y_1,y_3\}$, which is the same as the case in \cref{fig:phase3 triangle}. 
The sharp error bound for the case in \cref{fig:phase3 cone} can be calculated using exactly \cref{eq:Hstar 3,eq:phase3} except the final result needs to be multiplied by $-1$. 

\section{Discussion} \label{sec:discussion}
Results in \cref{sec:phase2,sec:phase3} provide the sharp error bound for bivariate linear interpolation and extrapolation for any configuration of affinely independent $\cY$ and $y_0$. 
Despite this, we are unable to prove the connection between the signs of $\{\mu_{ij}\}$ and whether $y_0$ is inside the shaded areas in \cref{fig:phase2}. 
The definition of $\{\mu_{ij}\}$ involves the eigenvectors of $G$, but eigendecomposition in general does not have closed form solutions, making it difficult to analyze the signs of $\{\mu_{ij}\}$. 
We tried without success to prove this connection using only the properties of eigenvectors. 

The problem of determining analytically when $\{\mu_{ij}\}$ are all non-negative becomes more difficult in higher dimension. 
Unlike triangles, which can only have obtuse angles at no more than one vertex, simplices in higher dimension can violate condition \cref{eq:acute simplex} in many ways. 
They can have $\langle y_j-y_i, y_k-y_i \rangle < 0$ at multiple vertices $y_i$ and at the same time for multiple $(j,k)$ for each $y_i$. 
While there can only be up to four disconnected subset of $\R^2$ where $\{\mu_{ij}\}$ has negative elements, our numerical experiment shows this number can go up to at least twenty for trivariate ($n=3$) linear extrapolation. 
It would be difficult to describe all those areas, let alone analyzing them. 

Ultimately, we hope to find a general formula for the sharp bound on the function approximation error of linear interpolation and extrapolation. 
We found this bound might be $\langle G, H^\star \rangle/2$, but the matrix $H^\star$ is not determined, as shown in the two definitions of $H^\star$ \cref{eq:Hstar,eq:Hstar 3}. 
The matrix $H^\star$ is tied to $\{\mu_{ij}\}$ in \cref{eq:mu1 -,eq:mu1 +}, and we believe even when there are negatives in $\{\mu_{ij}\}$, it is still tied in the same manner to a version of $\{\mu_{ij}\}$ that is modified to be all non-negative. 
In fact, \cref{eq:mu0 -,eq:mu0 +,eq:mu1 -,eq:mu1 +} all hold true under the setting of \cref{thm:phase3} if $H^\star$ is defined as \cref{eq:Hstar 3} and $\{\mu_{ij}\}$ is defined as 
\[ \begin{aligned} 
\mu_{10} &= 1-\ell_2, 
&\mu_{13} &= -\ell_3, 
&\mu_{20} &= \ell_2, 
&\mu_{23} &= 0, 
\end{aligned} \]
which are the coefficients in \cref{eq:phase3 sum}. 
However, for now we do not know if a general formula for the sharp bound exists or not. 
The best general formula we can find is \cref{eq:phase1}, which is always a valid bound but not always sharp. 

Comparing to the condition \cref{eq:Lipschitz}, it is more customary in approximation theory literature to assume the function $f$ twice differentiable on some $Q \subseteq \R^n$, and use $\sup_{u\in Q} \|\nabla^2 f(u)\|$ in place of the Lipschitz constant $L$, where the norm is the spectral norm. 
For example, $Q$ is set to a star-shaped subset of $\R^n$ in \cite{ciarlet1972general} and to $\conv(\cY)$ in \cite{waldron1998error}. 
For our result \cref{eq:phase1}, $Q$ at least needs to cover (almost everywhere) the star-shaped set $\cup_{i=0}^{n+1} \{\alpha y_i + (1-\alpha)w:~ 0\le\alpha\le1\}$. 
For \cref{eq:phase2}, we need $Q$ to cover 
\[ \bigcup_{(i,j)\in\cI_+\times\cI_-} 
\left( \begin{aligned} 
&\{\alpha y_i + (1-\alpha)[(u_i+u_j)/2+H^\star(u_i-u_j)/(2L)]:~ 0\le\alpha\le1 \} \\
&\qquad \cup \{\alpha y_j + (1-\alpha)[(u_i+u_j)/2+H^\star(u_i-u_j)/(2L)]:~ 0\le\alpha\le1 \}
\end{aligned} \right), 
\] 
where $H^\star$ is defined as \cref{eq:Hstar}. 
For \cref{eq:phase3}, we need 
\[ Q \supseteq \{\alpha y_2 + (1-\alpha)w:~ 0\le\alpha\le1 \} \cup \{\alpha y_3 + (1-\alpha)w:~ 0\le\alpha\le1 \}, 
\] 
where $w$ is defined as \cref{eq:phase3 w}. 
Considering these two assumptions have similar effect, we opted for the simpler \cref{eq:Lipschitz} in our analyses and left the other for the discussion here. 

This paper does not contain any analysis on the gradient approximation error, which is worth investigating, considering linear interpolation is sometimes used for approximating the gradient rather than the function. 
However, we find such problem hard to solve or even define. 
If the error is measured as usual by the Euclidean norm of the difference, then the problem of finding its sharp bound can be formulated in a way similar to \cref{prob:D}: 
\begin{equation} \everymath{\displaystyle} \begin{array}{ll}
    \max_{f} &\|\nabla m(y_0) - \nabla f(y_0)\| \\
    \text{s.t. } &f \in C_L^{1,1}(\R^n). 
\end{array}
\end{equation}
No matter how the constraint is handled, the objective of this problem is to maximize a convex function, making it a much less tractable nonconvex problem. 
The problem can become tractable if the error is measured differently, but whether the measure is meaningful depends on the application. 
Some results regarding this error can be found in \cite{handscornb1995errors, cao2005error, berahas2022theoretical} if anyone is interested.

\section*{Acknowledgment}
We would like to acknowledge the help from Dr. Xin Shi and Yunze Sun in solving \eqref{prob:quadratic}. 
We would also like to thank Dr. Shuonan Wu for carefully reading this paper and discussing its relationship to finite element methods. 

\bibliographystyle{siamplain}
\bibliography{references}

\begin{thebibliography}{10}

\bibitem{berahas2022theoretical}
{\sc A.~S. Berahas, L.~Cao, K.~Choromanski, and K.~Scheinberg}, {\em A
  theoretical and empirical comparison of gradient approximations in
  derivative-free optimization}, Foundations of Computational Mathematics, 22
  (2022), pp.~507--560.

\bibitem{cao2005error}
{\sc W.~Cao}, {\em On the error of linear interpolation and the orientation,
  aspect ratio, and internal angles of a triangle}, SIAM Journal on Numerical
  Analysis, 43 (2005), pp.~19--40.

\bibitem{ciarlet1972general}
{\sc P.~G. Ciarlet and P.-A. Raviart}, {\em General {L}agrange and {H}ermite
  interpolation in {$\R^n$} with applications to finite element methods},
  Archive for Rational Mechanics and Analysis, 46 (1972), pp.~177--199.

\bibitem{DFO_book}
{\sc A.~R. Conn, K.~Scheinberg, and L.~N. Vicente}, {\em Introduction to
  derivative-free optimization}, SIAM, 2009.

\bibitem{davis1975book}
{\sc P.~J. Davis}, {\em Interpolation and Approximation}, Courier Corporation,
  1975.

\bibitem{handscornb1995errors}
{\sc D.~Handscornb}, {\em Errors of linear interpolation on a triangle}, tech.
  report, Oxford University Computing Laboratory, 1995.

\bibitem{horn2012matrix}
{\sc R.~A. Horn and C.~R. Johnson}, {\em Matrix Analysis}, Cambridge university
  press, 2012.

\bibitem{Nesterov_book}
{\sc Y.~Nesterov}, {\em Introductory Lectures on Convex Optimization: A Basic
  Course}, vol.~87, Springer Science \& Business Media, 2013.

\bibitem{powell2001lagrange}
{\sc M.~Powell}, {\em On the {L}agrange functions of quadratic models that are
  defined by interpolation}, Optimization Methods and Software, 16 (2001),
  pp.~289--309.

\bibitem{powell1994direct}
{\sc M.~J. Powell}, {\em A direct search optimization method that models the
  objective and constraint functions by linear interpolation}, in Advances in
  optimization and numerical analysis, Springer, 1994, pp.~51--67.

\bibitem{stampfle2000optimal}
{\sc M.~St{\"a}mpfle}, {\em Optimal estimates for the linear interpolation
  error on simplices}, Journal of Approximation Theory, 103 (2000), pp.~78--90.

\bibitem{sylvester1852xix}
{\sc J.~J. Sylvester}, {\em Xix. a demonstration of the theorem that every
  homogeneous quadratic polynomial is reducible by real orthogonal
  substitutions to the form of a sum of positive and negative squares}, The
  London, Edinburgh, and Dublin Philosophical Magazine and Journal of Science,
  4 (1852), pp.~138--142.

\bibitem{waldron1998error}
{\sc S.~Waldron}, {\em The error in linear interpolation at the vertices of a
  simplex}, SIAM Journal on Numerical Analysis, 35 (1998), pp.~1191--1200.

\end{thebibliography}
\end{document}